\documentclass[reqno,11pt]{amsart}  
\usepackage[utf8]{inputenc}

\usepackage[colorinlistoftodos]{todonotes}
\usepackage{pgfplots} 
\usepackage{subfigure}
\usepackage{amsmath}
\usepackage{amssymb} 
\usepackage{mathtools}   % loads »amsmath«
\usepackage{graphicx}
\usepackage{graphics}  
\usepackage{color}
\usepackage{verbatim} 
\usepackage{bbold}
\usepackage[normalem]{ulem} 
\usepackage[mathscr]{eucal}
\usepackage{upgreek}
\usepackage{enumerate} 
\usepackage{cite}
\usepackage{amsmath}
\usepackage{amsfonts}
\usepackage{amssymb}
\usepackage{bbm}
\usepackage{cancel}
\usepackage{graphicx}
\usepackage{multicol}
\usepackage[utf8]{inputenc}
\usepackage{framed}
\usepackage{setspace}
\usepackage{amsthm}
\usepackage{hyperref}
\usepackage{caption}
\usepackage{subcaption}
\usepackage{bbm}
\usepackage{graphicx}
\usepackage{tikz}
\usepackage[titletoc]{appendix}
\numberwithin{equation}{section}  

\usepackage[refpage,noprefix]{nomencl}
\usepackage{nomencl} 

\newcommand{\ssup}[1] {{\scriptscriptstyle{{#1}}}}
\newcommand{\gk}[1]{\left\{#1\right\}}

\newcommand{\ek}[1]{\left[#1\right]}
\newcommand{\rk}[1]{\left(#1\right)}
 \newcommand{\floor}[1]{\left\lfloor #1 \right\rfloor}
\newcommand{\ceil}[1]{\left\lceil #1 \right\rceil}
\newcommand{\hk}[1]{^{\ssup{(#1)}}}

\newcommand{\I}   {{\mathcal I }}

\newcommand{\Rcal}   {{\mathcal R }}

\newcommand{\Z}     {\mathbb{Z}} 
\newcommand{\N}     {\mathbb{N}} 
\renewcommand{\P}   {\mathbb{P}} 
 
\newcommand{\E}     {\mathbb{E}}

\newcommand{\up}{\mathrm{up}}

 \newcommand{\ex}{{\rm e}} 
 \renewcommand{\d}{{\rm d}} 
 
\renewcommand{\L}{\Lambda}
\renewcommand{\l}{\lambda}
\newcommand{\e}{\varepsilon}
\newcommand{\1}{\!\mathbbm{1}\!}

\renewcommand{\L}{\Lambda}

\renewcommand{\P}{\mathbb{P}}

\newtheorem{theorem}{Theorem}
\newtheorem{lemma}{Lemma}[section]

\newcommand{\eps}{\varepsilon}

\newcommand{\Range}{\Rcal}
\newcommand{\Centrange}{\overline{\Range}}

\newcommand{\ind}[1]{{\mathbf 1}{\{#1\}}}

\title{On large deviations for the range of a two-dimensional random walk}
\author{Serguei Popov$^{1}$, Quirin Vogel$^{2}$}
\date{\today}

\begin{document}

\maketitle
\begin{abstract}
In this note, we compute the probability that a two-dimensional symmetric random walk visits more vertices than expected, for deviations on scales between the mean behavior and linear growth. 
\end{abstract}

\centerline{\textit{$^1$ Centro de Matem\'atica, University of 
Porto, Porto, Portugal}}
\centerline{\textsc{
serguei.popov@fc.up.pt}}
\centerline{\textit{$^2$ University of Klagenfurt,
Department of Statistics, Klagenfurt, Austria}}
\centerline{\textsc{quirin.vogel@aau.at}}
\bigskip


\section{Introduction and results}
Let $(X_i)_{i\geq 1}$ be i.i.d. symmetric random vectors in $\Z^2$ with mean zero and finite second moment; we also assume that they are not supported on a one-dimensional subspace. 
Let us define the random walk~$(S_n)_{n\geq 1}$ as $S_n=\sum_{i=1}^n X_i$ for $n\ge 1$. For this two-dimensional random walk, denote the range and its deviation by
\begin{equation}
\Range_n=\sum_{x\in\Z^2}\1\gk{H_x\le n}\qquad\textnormal{and}\qquad \Centrange_n=\Range_n-\E\ek{\Range_n}\, ,
\end{equation}
where $H_x=\inf\gk{n\ge 0\colon S_n=x}$ is the entrance time to~$x$. Denote the covariance matrix of~$X_1$ by~$\Gamma$. Then we have, as $n\to\infty$,
\begin{equation}\label{eq:exp_expans}
    \E\ek{\Range_n}=\frac{2\pi n\sqrt{\det{\Gamma}}}{\log n}+\frac{2\pi n\sqrt{\det{\Gamma}}}{\log^2n}\rk{1+o(1)}\, ,
\end{equation}
see \cite[Theorem 6.9]{le1991range} (we write $\log^\alpha n$ for $\rk{\log(n)}^\alpha$, for better legibility). This result goes back to \cite{dvoretzky1951some}, where the first order was obtained for the simple random walk, where
$ 2\pi\sqrt{\det{\Gamma}}=\pi$.
For convenience, we henceforth denote $e_1=2\pi\sqrt{\det{\Gamma}}$ and $r_n=\E\ek{\Range_n}$.

A nonstandard central limit theorem was proved in \cite{le1986proprietes} for the random walk range in two dimensions, see also \cite[Theorem~5.4.3]{chen2010random}, showing that
\begin{equation}\label{eq:CLT}
    \frac{\log^2 n}{n}\Centrange_n\xrightarrow[n\to\infty]{(\mathrm{d})}-(2\pi)^2\sqrt{\det\Gamma }\,\gamma_1\, ,
\end{equation}
where $\gamma_1$ is the renormalized self-intersection local time for the Brownian motion in two dimensions up to time 1.

We now formulate our main result:
\begin{theorem}
\label{t_main}
There exist $C_{\up}, C_{\mathrm{low}}>0$ such that for all $\theta_n\ge 1$ with $\theta_nr_n\le n$, we obtain that for all $n$
\begin{equation}
\label{eq:main}
    \exp\rk{-C_{\mathrm{low}}n^{1-\theta_n^{-1}}}\le \P\rk{\Range_n\ge \theta_nr_n}\le \exp\rk{-C_{\mathrm{up}}n^{1-\theta_n^{-1}}}\, .
\end{equation}
\end{theorem}
Note that in the special case where $\theta_n=(1+\delta)$ for $\delta>0$, we obtain large deviation estimates on the scale of the mean, with the factor $n^{1-\theta_n^{-1}}=n^{\frac{\delta}{\delta+1}}$ in the exponent. Large deviations at the linear scale (i.e., $\theta_n=\delta\log n$) were first studied in~\cite{hamana2001large} for $d\ge 2$, with asymptotic bounds. In \cite{bass2009moderate}, the authors studied the deviations of $\Centrange_n$ on scales much smaller than $r_n$ for $d=2$, also without matching constants.

We remark that large deviations in the downwards direction (i.e., the probability that the range is less than its expectation) were investigated in \cite{liu2021large} for $d=2$ on the scale of the mean and in~\cite{donsker1979number} for smaller scales.

In the remaining part of the paper, we prove Theorem~\ref{t_main}: first the upper bound in~\eqref{eq:main} (Section~\ref{s_upper_proof}), then the lower bound (Section~\ref{s_lower_proof}).

\section{Upper bound}
\label{s_upper_proof}
Let us prove the following (slightly more general) result:
\begin{lemma}
There exists $C_\up>0$ such that for all $n$ and all $\theta_n\ge 1$ with $r_n\theta_n\le n$, we have that
\begin{equation}
    \P\rk{\Range_n\ge \theta_nr_n}\le \exp\rk{-C_\up n^{1-\theta_n^{-1}}}\, ,
\end{equation}
If $\theta_n=o\rk{\log n}$, then we can choose $C_\up=\widetilde{\Lambda}(1+o(1))$
with
    \begin{equation}
    \label{def_tilLambda}
    \widetilde{\Lambda}=\ex^{-(\L'(b_0)-1)}b_0\, ,
    \end{equation}
    where $\Lambda$ is the logarithmic moment generating function of $-\gamma_1$, with $\gamma_1$ the renormalized self-intersection local time in the unit interval of the Brownian motion and $b_0$ solves $\Lambda(b)=b\rk{\Lambda'(b)-1}$.
In particular, for every $\delta>0$
    \begin{equation}
    \label{use_tilLambda}
    \limsup_{n\to\infty}\frac{1}{n^{\frac{\delta}{\delta+1}}}\log\P\rk{\Range_n\ge (1+\delta)r_n}\le -\widetilde{\Lambda}\, .
    \end{equation}
    \end{lemma}
\begin{proof}
Without loss of generality, we can assume that $\theta_n<\e_0\log n$ for $\e_0\in (0,1)$ arbitrarily small but fixed, since for some $\psi\colon (0,1]\to (0,\infty)$
\begin{equation}
    \P\rk{\Range_n\ge \e_0 r_n\log n}\le \P\rk{\Range_n\ge \e_0 e_1 n/2}=\ex^{-\psi(\e_0/2)n\rk{1+o(1)}}\, ,
\end{equation}
by \cite[Theorem 1]{hamana2001large}. Define now $\alpha_n=\frac{1}{\theta_n}\le 1$ as well as
\begin{equation}
    m=\floor{\ex^{\beta+1}n^{\alpha_n}}\qquad\textnormal{and}\qquad M=\ceil{n/m}\, .
\end{equation}
Note that if $\theta_n=o(\log n)$, then $m$ diverges to infinity. Also, we define
\begin{equation}
    \Range_{a,b}=\mathop{\mathrm{card}}\{S_a,\ldots,S_{b-1}\}
    %\#\gk{x\in\Z^2\colon \exists k\in [a,b)\cap \N\textnormal{ with }S_k=x}\, ,
\end{equation}
to be the number of sites visited between time $a$ up to time~$b-1$ and, similarly,
let $\Centrange_{a,b} = \Range_{a,b} - \E\ek{\Range_{a,b}}=\Range_{a,b}-r_{b-a}$. Naturally, we have
\begin{equation}
    \Range_n\le \sum_{i=1}^M\Range_{(i-1)m,im}\, ,
\end{equation}
and hence
\begin{equation}\label{eq:rangePart}
    \P\rk{\Range_n\ge \theta_n r_n}\le \P\rk{\sum_{i=1}^M\Centrange_{(i-1)m,im}\ge \theta_n r_n-Mr_m}\, .
\end{equation}
By~\eqref{eq:exp_expans}, there exists $\delta_m\to 0$ (as $m\to\infty$), such that
\begin{equation}
    r_m\le \frac{e_1 m}{\log m}+\frac{e_1 m}{\log^2m}(1+\delta_m/2)\le \frac{e_1 \ex^{\beta+1}n^{\alpha_n}}{\alpha_n\log n}-\frac{e_1 \beta \ex^{\beta+1}n^{\alpha_n}}{\alpha_n^2\log^2n}(1+\delta_m)\, ,
\end{equation}
using the expansion
\begin{equation}
    \frac{1}{\log m}= \frac{1}{\alpha_n\log n+(\beta+1)}\rk{1+o(1)}=\frac{1}{\alpha_n\log n}-\frac{\beta+1}{\alpha_n^2\log^2 n}\rk{1+o(1)}\, .
\end{equation}
Therefore, we obtain
\begin{equation}
    Mr_m\le \theta_nr_n-\beta\frac{\theta_n^2r_n}{\log n}\rk{1+2\delta_m}\, .
\end{equation}
This implies by~\eqref{eq:rangePart} that
\begin{equation}
    \P\rk{\Range_n\ge \theta_n r_n}\le \P\rk{\sum_{i=1}^M\Centrange_{(i-1)m,im}\ge \beta \frac{\theta_n^2r_n}{\log n}\rk{1+2\delta_m} }\, .
\end{equation}
One can similarly verify that
\begin{equation}
    \beta \frac{\log^2 m}{m}\frac{\theta_n^2r_n}{\log n}\rk{1+2\delta_m} \ge \beta M\rk{1+\delta_m}\, .
\end{equation}
Then, the exponential Chebyshev's inequality implies that
\begin{align}
    \P\rk{\Range_n\ge \theta_n r_n}&\le \P\rk{\sum_{i=1}^M\frac{\log^2 m}{m}\Centrange_{(i-1)m,im}\ge \beta M(1+\delta_m)}\\ 
    &\le \exp\rk{-\lambda \beta M\rk{1+\delta_m}+M\log\E\ek{\ex^{\lambda \widetilde{\Range}_m}}}\, ,
\end{align}
where $\widetilde{\Range}_m$ has the law of
\begin{equation}
    \widetilde{\Range}_m\stackrel{\mathrm{(d)}}=\frac{\log^2 m}{m}\Centrange_{(i-1)m,im}\, .
\end{equation}
Note that we can furthermore choose $\delta_m$ such that for fixed $\lambda>0$ 
\begin{equation}
    \log\E\ek{\ex^{\lambda \widetilde{\Range}_m}}\le \Lambda(\lambda)(1+\delta_m)\, ,
\end{equation}
where $\Lambda$ is the logarithmic moment generating function of $-\gamma_1$, the renormalized self-intersection local time of the Brownian motion, see~\cite{bass2009moderate} shortly before~(3.3). Choose now $\lambda=\lambda_0$ maximizing $\ek{\beta\lambda-\Lambda(\l)}$, i.e., such that $\ek{\beta\lambda-\Lambda(\l)}$ becomes the large-deviation rate of $-\gamma_1$ at point~$\beta$. Write $\Lambda^*(\beta)=\lambda_0-\Lambda(\lambda_0)$. We then get that
\begin{equation}
      \P\rk{\Range_n\ge \theta_n r_n}\le \exp\rk{- \Lambda^*(\beta) M\rk{1+3\delta_m}}\, .
\end{equation}
Define now 
\begin{equation}
        \widetilde{\Lambda}=\inf_{\beta>0}\gk{\ex^{-(\beta+1)}\Lambda^*(\beta)}\, ,
\end{equation}
and note that this infimum is achieved at the point $\beta_0$ where
\begin{equation}
       \frac{\d}{\d \beta}\Lambda^*(\beta)=\Lambda^*(\beta)\, .
\end{equation}
Let us denote $b_0=\tfrac{\d}{\d \beta}\Lambda^*(\beta_0)$. By the definition on convex conjugate, we have $\L'(b_0)=\beta_0$ and $\Lambda^*(\beta_0)=b_0\beta_0-\Lambda'(b_0)$. This then yields that~$b_0$ solves the equation $\Lambda(b)=b\rk{\Lambda'(b)-1}$. This concludes the proof.
\end{proof}

\section{Lower bound}
\label{s_lower_proof}
We want to prove that there is $C>0$ such that for all~$n$ and all $\theta_n\geq 1$
such that $\theta_n r_n\leq n$ we have
\begin{equation}
 \label{LD_lower}
  \P(\Range_n\geq \theta_n r_n) \geq \exp(-C n^{1-\theta_n^{-1}}).
\end{equation}
First, assume that $\theta_n<\eps_0\log n$ for some small~$\eps_0$
(otherwise, \eqref{LD_lower} trivially holds because of the obvious strategy
``force the random walk to increase the 1st coordinate on every step'', which leads to $\P(\Range_n=n)\geq c_1^n$ for some $c_1>0$).
For a large constant~$\beta>0$ (to be chosen later) define
\[
 m=\exp\Big(\frac{\log n}{\theta_n} - \beta\Big)
  \quad \text{ and } \quad
  M = \frac{n}{m} = \exp\big(\beta+(1-\theta_n^{-1})\log n\big).
\]
In the following, for simplicity, we do the calculations as if~$m$ and~$M$
were integers; it is straightforward to check that the calculations in the general case are essentially the same (see the previous section for details). 
Note that, to prove~\eqref{LD_lower}, it suffices to show that
$\P(\Range_n\geq \theta_n r_n) \geq c_2^M$ for some positive~$c_2$.

For $\ell\in\N$, 
write $S_\ell\hk{1}$ for the first coordinate of the random walker (at time~$\ell$). For $k=1,\ldots,M$ define the (independent and same-probability) events
\begin{align*}
B_k &= \big\{S_{(k-1)m+j}\hk{1}-S_{(k-1)m}\hk{1}\in (-m^{1/2},3m^{1/2}) \text{ for }
 j=0,\ldots,m-1, \\
 & \qquad \qquad \qquad \qquad \qquad \qquad
 \text{ and }S_{km-1}\hk{1}-S_{(k-1)m}\hk{1}\in (2m^{1/2},3m^{1/2})\big\}.
\end{align*}
See on Figure~\ref{f_strategy2} an illustration of events $B_1,B_2,B_3$;
in particular, note that, when these events occur, the two bold pieces of the trajectory
(i.e., those corresponding to~$B_1$ and~$B_3$) cannot intersect.
\begin{figure}
\begin{center}
\includegraphics{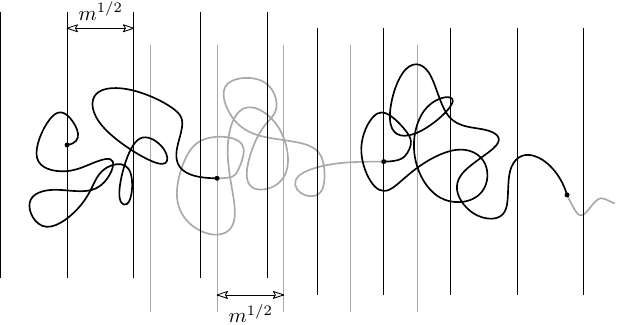}
\caption{Definition of the events $B_k$ and the strategy for the lower bound.}
\label{f_strategy2}
\end{center}
\end{figure}
Note that, by the Donsker's invariance principle (see e.g. \cite[Theorem 3.4.2]{lawler2010random}), there exists $h_0>0$ such that,
for all $m\geq 8$
\begin{equation}
\label{prob_Bk}
 \P(B_k)\geq h_0.
\end{equation}
Next, for $k=1,\ldots,M$ let us abbreviate $\Range^{(k)}:=\Range_{(k-1)m,km-1}$ and define the (again, independent and same-probability) events 
\[
E_k = \Big\{\Range^{(k)}\geq r_m \Big(1-\frac{\beta/2}{\log m}\Big)\Big\}.
\]
Note that~\eqref{eq:CLT} implies that, with some $f(\beta)\to 0$ as $\beta\to\infty$
\begin{equation}
\label{ocenka_Ek}
 \P(E_k) \geq 1-f(\beta).
\end{equation}
Then, define 
\[
\I^{(k,k+1)} = \mathop{\mathrm{card}}\big(\{S_{(k-1)m},\ldots,S_{km-1}\}
\cap \{S_{km},\ldots,S_{(k+1)m-1}\}\big)
\]
to be the sizes of the intersections of ``neighbouring'' ranges, and let
\[
I_{k,k+1} = \Big\{\I^{(k,k+1)}\leq r_m \frac{\beta/6}{\log m}\Big\}.
\]
We note that $\I^{(k,k+1)}=\Range^{(k)}+\Range^{(k+1)}-\Range_{(k-1)m,(k+1)m-1}$, so
\[
\E \ek{\I^{(k,k+1)}} = 2r_m-r_{2m} = r_m\frac{2\log 2}{\log m}(1+o(1)).
\]
Therefore, Chebyshev's inequality implies that
\begin{equation}
\label{ocenka_Ik_k+1}
 \P(I_{k,k+1}) \geq 1 - \frac{2\log 2}{\beta/6}(1+o(1)).
\end{equation}

Now, on the event $B:=B_1\cap\ldots \cap B_M$ only neighbouring ranges can intersect (again, see Figure~\ref{f_strategy2}), so, on~$B$ we have
\begin{equation}
\label{sum_ranges}
\Range_n = \Range^{(1)} + \cdots + \Range^{(M)} - (\I^{1,2}+ \cdots + \I^{M-1,M}).
\end{equation}
Then, it holds that (note that $\log n = \theta_n(\beta+\log m)$)
\[
\frac{r_n}{r_m} = M \frac{\log m}{\theta_n(\beta+\log m)} \Big(1 - \frac{1}{\log n} + \frac{1}{\log m}(1+o(1))\Big),
\]
so 
\begin{equation}
\label{theta_n_r_n}
\theta_n r_n \leq M r_m\Big(1 - \frac{\beta}{\log m} (1+o(1))\Big).
\end{equation}
Denote $E:=E_1\cap\ldots \cap E_M$ and $I:=I_{1,2}\cap\ldots\cap I_{M-1,M}$. 
Then, on $B\cap E\cap I$ we have, due to~\eqref{sum_ranges}
\[
\Range_n \geq M r_m \Big(1-\frac{\beta/2}{\log m}-\frac{\beta/6}{\log m}\Big)
  = M r_m \Big(1-\frac{2\beta/3}{\log m}\Big),
\]
and so, due to~\eqref{theta_n_r_n}, on $B\cap E\cap I$ 
the event $\{\Range_n\geq \theta_n r_n\}$ occurs (at least if~$m$ is large enough).

We are thus left with the task of finding a lower bound for
$\P(B\cap E\cap I) = \P(B)\P(E\cap I \mid B)$.
Denote $\eta_k=\ind{E_k\cap E_{k+1}\cap I_{k,k+1}}$, and let $\P^*(\cdot) = \P(\,\cdot\mid B)$.
Note that $\{\eta_k = 1\}$ is independent of $(B_\ell, \ell\neq k, k+1)$,
and so~\eqref{prob_Bk}, \eqref{ocenka_Ek} and~\eqref{ocenka_Ik_k+1} imply
that
\begin{equation}
\label{eta=1}
 \P^*(\eta_k=1) = \P(E_k\cap E_{k+1}\cap I_{k,k+1} \mid B_k\cap B_{k+1}) 
  \geq 1 - \frac{2f(\beta)+ \frac{2\log 2}{\beta/6}(1+o(1))}{h_0^2} > \frac{3}{4}
\end{equation}
if $\beta$ is large enough.
Now, $(\eta_1,\ldots,\eta_M)$ is a $1$-dependent random sequence under~$\P^*$
(that is, $(\eta_k, k\in A_1)$ and $(\eta_k, k\in A_2)$ are $\P^*$-independent when
no element of~$A_1$ is a neighbour of an element of~$A_2$),
and so Theorem~0.0~(i) of~\cite{LSS97} implies that 
\[
\P(E\cap I \mid B) = \P^*[\eta_1=\ldots=\eta_M=1] \geq (1/4)^M.
\]
Since we also have $\P(B)\geq h_0^M$, this implies~\eqref{LD_lower} and thus concludes the proof of Theorem~\ref{t_main}.

\section*{Acknowledgments}
The authors would like to thank the organizers of the CIRM research school 3451 \textit{Marches al\'eatoires: applications et interactions} during which they had the opportunity to discuss this problem. SP was partially supported by
CMUP, member of LASI, which is financed by national funds
through FCT (Funda\c{c}\~ao para a Ci\^encia e a Tecnologia, I.P.) 
under the project with reference UID/00144/2025.

\bibliographystyle{alpha}
\bibliography{references.bib}

@book{chen2010random,
  title={Random walk intersections: large deviations and related topics},
  author={Chen, X.},
  number={157},
  year={2010},
  publisher={American Mathematical Soc.}
}

@inproceedings{dvoretzky1951some,
  title={Some problems on random walk in space},
  author={Dvoretzky, A. and Erdős, P.},
  booktitle={Proc. Second Berkeley Symp. Math. Statist. Probab},
  pages={353--367},
  year={1951}
}

@article{le1986proprietes,
  title={Propri{\'e}t{\'e}s d'intersection des marches al{\'e}atoires: I. {C}onvergence vers le temps local d'intersection},
  author={Le Gall, J-F.},
  journal={Communications in mathematical physics},
  volume={104},
  number={3},
  pages={471--507},
  year={1986},
  publisher={Springer}
}

@article{liu2021large,
  title={Large deviations of the range of the planar random walk on the scale of the mean},
  author={Liu, J. and Vogel, Q.},
  journal={Journal of Theoretical Probability},
  volume={34},
  number={4},
  pages={2315--2345},
  year={2021},
  publisher={Springer}
}

@article{donsker1979number,
  title={On the number of distinct sites visited by a random walk},
  author={Donsker, M. and Varadhan, S.},
  journal={Communications on Pure and Applied Mathematics},
  volume={32},
  number={6},
  pages={721--747},
  year={1979},
  publisher={Wiley Online Library}
}

@book{lawler2010random,
  title={Random walk: a modern introduction},
  author={Lawler, G. and Limic, V.},
  volume={123},
  year={2010},
  publisher={Cambridge University Press}
}

@article{hamana2001large,
  title={A large-deviation result for the range of random walk and for the Wiener sausage},
  author={Hamana, Y. and Kesten, H.},
  journal={Probability theory and related fields},
  volume={120},
  number={2},
  pages={183--208},
  year={2001},
  publisher={Springer}
}

@article{le1991range,
  title={The range of stable random walks},
  author={Le Gall, J-F. and Rosen, J.},
  journal={The Annals of Probability},
  pages={650--705},
  year={1991},
  publisher={JSTOR}
}

@book{bass2009moderate,
  title={Moderate deviations for the range of planar random walks},
  author={Bass, R. and Chen, X. and Rosen, J.},
  year={2009},
  publisher={American Mathematical Soc.}
}

@article{LSS97,
 ISSN = {00911798, 2168894X},
 author = {Liggett, T.M. and Schonmann, R.H. and Stacey, A.M.},
 journal = {The Annals of Probability},
 number = {1},
 pages = {71--95},
 publisher = {Institute of Mathematical Statistics},
 title = {Domination by product measures},
 urldate = {2026-02-04},
 volume = {25},
 year = {1997}
}
\end{document}